\documentclass[a4paper,11pt]{amsart}
\usepackage{amsmath}
\usepackage{amsthm}
\usepackage{amssymb}
\usepackage{amsfonts}
\usepackage{enumerate}
\usepackage[all]{xy}
\usepackage{graphicx}
\usepackage{multicol}
\usepackage{fancyhdr}
\usepackage{tabularx} 
\usepackage{paralist} 

\usepackage[utf8]{inputenc}

\normalsize

\newtheoremstyle{theoremstyle}
  {10pt}      
  {5pt}       
  {\itshape}  
  {}          
  {\bfseries} 
  {:}         
  {.5em}      
  {}          

\newtheoremstyle{examplestyle}
  {10pt}      
  {5pt}       
  {}          
  {}          
  {\bfseries} 
  {:}         
  {.5em}      
  {}          

\theoremstyle{theoremstyle}
\newtheorem{theorem}{Theorem}[section]
\newtheorem*{theorem*}{Theorem}
\newtheorem{lemma}[theorem]{Lemma}

\newtheorem*{proposition*}{Proposition}
\newtheorem{corollary}[theorem]{Corollary}
\newtheorem*{corollary*}{Corollary}

\newtheorem*{question*}{Question}

\theoremstyle{examplestyle}

\newtheorem{definition*}{Definition}
\newtheorem{remark}[theorem]{Remark}
\newtheorem{remark*}{Remark}

\newtheorem{convention*}{Convention}

\newtheorem{notation*}{Notation}
\newtheorem*{tha*}{Theorem A}
\newtheorem*{thb*}{Theorem B}
\newtheorem*{thc*}{Theorem C}
\newtheorem*{thd*}{Theorem D}
\newcommand{\comment}[1]{}

\newcommand{\nor}{\mathrm{nor}}

\newcommand{\N}{\mathbb{N}}
\newcommand{\Z}{\mathbb{Z}}
\newcommand{\R}{\mathbb{R}}

\newcommand{\Sph}{\mathbb{S}}

\DeclareMathOperator{\tr}{trace}

\newcommand{\tp}{^{\mathrm{T}}}

\newcommand{\bmat}{\left(\begin{smallmatrix}}
\newcommand{\emat}{\end{smallmatrix}\right)}

\newcommand{\SO}{\mathrm{SO}}

\newcommand{\SU}{\mathrm{SU}}

\newcommand{\Gtwo}{\mathrm{G}_2}

\begin{document}

\title[Harmonic self-maps of $\SU(3)$]{Harmonic self-maps of $\SU(3)$}
\author{Anna Siffert$^1$}%
\address{Max-Planck-Institut f\"ur Mathematik\\
 Vivatsgasse 7\\
53111 Bonn\\
Germany}
\email{siffert@mpim-bonn.mpg.de}

\begin{abstract}
By constructing solutions of a singular boundary value problem we prove the existence of a countably infinite family of harmonic self-maps of $\SU(3)$ with non-trivial, i.e. $\neq 0,\pm 1$, Brouwer degree.

\end{abstract}

\maketitle
\section*{Introduction}

\thispagestyle{empty}
\footnotetext[1]{The author would like to thank Deutsche Forschungsgemeinschaft for supporting this work with the grant SI 2077/1-1.
Furthermore, I would like to thank the Max Planck Institute for Mathematics for the support and the excellent working conditions.}
The energy of a smooth map $\varphi:M\rightarrow N$ between two Riemannian manifolds $(M,g)$ and $(N,h)$ is defined by
$$E(\varphi)=\int_M\lvert d\varphi\rvert^2\omega_g,$$
where $\omega_g$ denotes the volume measure on $M$.
A smooth map is called harmonic if it is a critical point of the energy functional, i.e., satisfies the Euler-Lagrange equations 
\begin{align*}
\tau(\varphi)=0,
\end{align*}
where $\tau(\varphi):=\mbox{trace}\nabla d\varphi$ is the so-called tension field of $\varphi$.
Finding solutions of this elliptic, semi-linear partial differential equation of second order is difficult in general. 

\medskip

By imposing suitable symmetry conditions, the Euler-Lagrange equations sometimes reduce to ordinary differential equations. 
This is the case for the following situation which is dealt with in this paper:
we consider
the cohomogeneity one action
\begin{align*}
\SU(3)\times\SU(3)\rightarrow\SU(3),\hspace{1cm} (A,B)\mapsto ABA^{\tp}
\end{align*}
of $G=\SU(3)$ on itself, whose principal isotropy group is given by $H=\SO(2)$. 
For any smooth map $r: ]0,\pi/2[\rightarrow\R$ we define the map
\begin{align*}
\psi_r:G/H\times ]0,\pi/2[\rightarrow G/H\times\R,\hspace{1cm} (gH,t)\rightarrow (gH,r(t)),
\end{align*}
which is equivariant with respect to the above action.
For these maps the Euler-Lagrange equations of the energy functional reduce to
\begin{align*}
\ddot r(t)=-\csc^22t\left(2\sin4t\cdot\dot r(t)+4\sin^2t\cdot\sin2r(t)-8\cos^3t\cdot\sin r(t)\right).
\end{align*}
We prove that each solution of this ordinary differential equation which satisfies
$r(0)=0$ and $r(\tfrac{\pi}{2})=(2\ell+1)\tfrac{\pi}{2}, \ell\in\Z$, yields a harmonic self-map of $\SU(3)$.
The above ordinary differential equation and boundary value problem are henceforth referred to as ODE and BVP, respectively. 

\smallskip

The goal of this paper is the construction of solutions of the BVP and the examination of their properties. Thereby we construct and examine harmonic 
self-maps of $\SU(3)$.

\subsection*{Brouwer degree}
The Brouwer degree of $\psi_r$ is determined in terms of $\ell$ only: from Theorem\,3.4 in \cite{puttmann} we deduce that for any solution $r$ of the BVP with $r(\tfrac{\pi}{2})=(2\ell+1)\tfrac{\pi}{2}$ the Brouwer degree of $\psi_r$ is given by
$\mbox{deg}\,\psi_{r}=2\ell+1$. 

\smallskip 

By an intricate examination of the BVP
we find restrictions for $\ell$ and thus for the possible Brouwer degrees of $\psi_r$. 

\begin{tha*}
\label{tea}
For each solution $r$ of the BVP we have $\mbox{deg}\,\psi_{r}\in\lbrace\pm 1, \pm 3, \pm5,\pm 7\rbrace$.
\end{tha*}

Numerical experiments indicate that for all solutions $r$ of the BVP the Brouwer degree of $\psi_r$ is $\pm 1$ or $\pm 3$, i.e., that the cases $\pm 5$ and $\pm 7$ do not arise.

\smallskip

These considerations and results can be found in Section\,\ref{sec2}.

\medskip

\subsection*{{Construction of solutions}}
In order to find solutions of the BVP we use a shooting method at the degenerate point $t=0$. 
This is possible since for each $v\in\R$ there exists a unique solution of the initial value problem at $t=0$.

\begin{thb*}
For each $v\in\R$ the initial value problem $r(t)_{\lvert t=0}=0, \dot r(0):=\tfrac{d}{dt}r(t)_{\lvert t=0}=v$
has a unique solution $r_v$.
\end{thb*}

We prove that we cannot increase the initial velocity $v$ arbitrarily without increasing the number
of intersections of $r_v$ and $\pi$, the so-called \textit{nodal number}.
This is one of the main ingredients for the proof of the following theorem.

\begin{thc*}
\label{tea}
For each $k\in\N$ there exists a solution of the BVP with nodal number $k$.
\end{thc*} 

Infinitely many of the solutions constructed in Theorem\,C have Brouwer degree of absolute value greater or equal to three.

\smallskip

These result are all contained in Section\,\ref{sec3}.

\medskip

\subsection*{{Limit configuration}}
We prove that the solutions of the BVP converge on every closed interval $I\subset (0,\tfrac{\pi}{2})$
against a limit configuration when the initial velocity goes to infinity:
we show that for large initial velocities $r_v$ becomes arbitrarily close to $\pi$ on $I$.  

\begin{thd*}
For every closed interval $I\subset (0,\tfrac{\pi}{2})$ and each $\epsilon>0$ there exists a velocity $v_0$ such that $\lvert r_v(t)-\pi\rvert<\epsilon$ for all $t\in I$ and $v\geq v_0$.
\end{thd*}

As a consequence we prove that the Brouwer degree of solutions $r_v$ of the BVP with \lq large\rq\hspace{0.06cm}initial velocity can only be $\pm 1$ or $\pm 3$. 

\smallskip

These results can be found in Section\,\ref{sec4}.

\bigskip

The paper is organized as follows: a short introduction to those aspects of harmonic maps needed in the present paper can be found in Section\,\ref{sec1}.
We provide the preliminaries in Section\,\ref{sec2}, where we in particular consider the Brouwer degree of the maps $\psi_r$ and prove Theorem\,A.
In Section\,\ref{sec3} we deal with the construction of solutions of the BVP and prove Theorem\,B and Theorem\,C.
Finally, in Section\,\ref{sec4} we investigate the behavior of those solutions of the initial value problem with large initial velocities and prove that they converge against a limit configuration, i.e., we show Theorem\,D.

\section{Harmonic maps between Riemannian manifolds}
\label{sec1}
Initiated by a paper of Eells and Sampson \cite{eells3}, the study of harmonic maps between Riemannian manifolds became an active research area, 
 see e.g. \cite{BC,eellsl,eells2,ga,smith} and the references therein. 
In this section we give a short and therefore incomplete introduction to harmonic maps between Riemannian manifolds.
The focus lies on those techniques, papers and results which we use as inspiration for some proofs contained in this work.
 For an elaborate introduction to harmonic maps we refer the reader to \cite{er}.

\smallskip

\textbf{Problem.}
In order to construct harmonic maps one has to find solutions $\varphi$ of the semi-linear, elliptic partial differential equation $\tau(\varphi)=0$.
Note that there is no general solution theory for these partial differential equations.

\medskip

\textbf{Central question.} The central question is whether every homotopy class 
 of maps between Riemannian manifolds admits a harmonic representative.
If the target manifold is compact and all its sectional curvatures are nonnegative Eells and Sampson gave a positive answer to this question.
However, if the target manifold also admits positive sectional curvatures the answer to this question is only known for some special cases. See e.g. \cite{ga2}
for a list of those homotopy groups of spheres which can be represented by harmonic maps.

\medskip

\textbf{Reduction by imposing symmetry.} By imposing symmetry conditions on the solution $\varphi$ of the partial differential equation $\tau(\varphi)=0$ one can sometimes reduce this problem to an easier problem, for example to finding solutions of an ordinary differential equation. 
For the general reduction theory we refer the reader to \cite{er}.

\smallskip

One special situation for which the Euler-Lagrange equations reduce to an ordinary differential equation is the following: the equivariant homotopy classes of equivariant self-maps of compact cohomogeneity one manifolds whose orbit space is a closed interval form an infinite family. 
In \cite{ps} P\"uttmann and the author reduced the problem of finding harmonic representatives of these homotopy classes to solving singular boundary value problems for nonlinear second order ordinary differential equations.\\
Note that the case under consideration, namely self-maps of $\SU(3)$ which are equivariant with respect to the cohomogeneity one action given in the introduction, clearly fits in this scheme. 

\medskip

\textbf{Harmonic maps between cohomogeneity one manifolds.}
We give a short survey of those results of \cite{ps} which are relevant for this paper.

\smallskip

\textit{Notation.} Let $G$ be a compact Lie group which acts with cohomogeneity one on the Riemannian manifold $(M,g)$ such that the orbit space is isometric to $\lbrack 0,1\rbrack$.
We denote by $\gamma$ a fixed normal geodesic.
The isotropy groups at the regular points are constant and will be denoted by $H$. Furthermore, let $M_{(H)}$ be the regular part of $M$
and $W$ the Weyl group, i.e., the subgroup of the elements of $G$ that leave $\gamma$ invariant modulo the subgroup of elements that fix $\gamma$ pointwise. Throughout this paper we assume that $\gamma$ is closed which is equivalent to the statement that $W$ is finite. Let $Q$ be a given biinvariant metric on $G$. Denote by $\mathfrak{g}$ and $\mathfrak{h}$ the lie algebras of $G$ and $H$, respectively, and let 
$\mathfrak{n}$ be the orthonormal complement of $\mathfrak{h}$ in $\mathfrak{g}$. Define the metric endomorphisms $P_t : \mathfrak{n} \to \mathfrak{n}$ by
$$Q(X_1, P_t\cdot X_2) = \langle X_1^{\ast},X_2^{\ast} \rangle_{\vert\gamma(t)},$$ where $X_i\in\mathfrak{g}$ and $X_i^{\ast}$
 is the associated action field on $M$.

\smallskip

\textit{Maps.} 
It was proved in \cite{puttmann} that the assignment $g\cdot\gamma(t)\mapsto g\cdot\gamma(kt)$ leads to a well defined smooth self-map of $M$, the $k$-map, if $k$ is of the form $k = j\lvert W\rvert/2+1$ where $j\in2\Z$. This is even true for any integer $j$ if the isotropy group at $\gamma(1)$ is a subgroup of the isotropy group at $\gamma(\lvert W\rvert/2+1)$.
In \cite{ps} P\"uttmann and the author examined the harmonicy of the so-called reparametrized $k$-maps $\psi : M_{(H)} \to M$ given by
\begin{gather*}
  \psi (g\cdot \gamma(t)) = g\cdot \gamma(r(t))
\end{gather*}
where $r : \;]0,1[\; \to \;]0,k[$ is a smooth function with $\lim_{t\to 0} r(t) = 0$ and $\lim_{t\to 1}r(t) = k$. 

\smallskip

\textit{Tension field.} 
In \cite{ps} it was shown that for the reparametrized $k$-maps the normal and the tangential component of the tension field 
are given by
\begin{gather*}
  \tau^{\nor}_{\vert\gamma(t)} = \ddot r(t) + \tfrac{1}{2}\dot r(t) \tr P_t^{-1}\dot P_t
    - \tfrac{1}{2} \tr P_t^{-1} (\dot P)_{r(t)}.
\end{gather*}
and
\begin{gather*}
\tau^{\tan}_{\vert\gamma(t)} = \bigl(P_{r(t)}^{-1} \sum_{i=1}^n [E_i,P_{r(t)}E_i]\bigr)^{\ast}_{\vert\gamma(r(t))},
\end{gather*}
respectively,
where $E_1,\ldots,E_n \in \mathfrak{n}$ are such that $E^{\ast}_{1\vert\gamma(t)},\ldots, E^{\ast}_{n\vert\gamma(t)}$ form an orthonormal basis of $T_{\gamma(t)}(G\cdot\gamma(t))$.

\begin{remark}
\begin{enumerate}
\item For the cohomogeneity one action $$\SU(3)\times\SU(3)\rightarrow\SU(3), (A,B)\mapsto ABA^{\tp},$$ the identity
$\tau^{\tan}=0$ holds trivially. 
It is proved in Section\,\ref{sec2} that the equation $\tau^{\nor}=0$ reduces to the ODE.
\item The Euler Lagrange equations associated to the cohomogeneity one action
\begin{align*}
{\SO(m_0\!+\!1)\times\SO(m_1\!+\!1)}\times\Sph^{m_0+m_1+1}\rightarrow\Sph^{m_0+m_1+1}, (A,B,v)\mapsto \left(\begin{smallmatrix}
A&0\\
0&B
\end{smallmatrix} \right)v
\end{align*}
are given by
\begin{align*}
\ddot r(t)=\left((m_1\!-\!m_0)\csc2t-(m_0\!+\!m_1)\cot2t\right)\dot r(t)-m_1\tfrac{\sin2r(t)}{2\cos^2t}+m_0\tfrac{\sin2r(t)}{2\sin^2t}.
\end{align*}
Each solution of this ordinary differential equation which satisfies
$r(0)=0$ and $r(\tfrac{\pi}{2})=(2\ell+1)\tfrac{\pi}{2}, \ell\in\Z$, yields a harmonic self-map of $\Sph^{m_0+m_1+1}$.
This boundary value problem was considered in \cite{si2}.
\end{enumerate}
\end{remark}

\section{Preliminaries}
\label{sec2}
This preparatory section is structured as follows: after proving in the first two subsections that each solution of the BVP yields a harmonic self-map of $\SU(3)$, we introduce the variable $x=\log\tan t$ in the third subsection.
It turns out that this variable is more convenient than the variable $t$ for several of our subsequent considerations.
In the fourth and fifth subsection we provide some restrictions for solutions $r$ of the BVP.
Finally, in the sixth subsection we prove Theorem\,A, i.e., that each solution $r$ of the BVP has Brouwer degree $\pm1, \pm3, \pm 5$ or $\pm 7$.

\smallskip

Throughout this section let $r$ be a solution of the ODE.

\subsection{Deduction of the ODE}
The expressions $\tau^{\nor}$ and $\tau^{\tan}$ depend on $P_t$ only, see Section\,\ref{sec1}. Hence, it is sufficient to determine this endomorphism
and plug it into these identities.
Let $\SU(3)$ be endowed with the metric $\langle A_1,A_2\rangle=\mbox{tr}(A_1\overline{A_2}^{\tp})$.
A normal geodesic $\gamma$ is given by
\begin{align*}
\gamma(t)=\left(\begin{smallmatrix}
\cos t&-\sin t&0\\
 \sin t & \cos t&0\\
0&0&1
\end{smallmatrix} \right).
\end{align*}
We consider the basis $\left\{\mu_{i}\right\}_{i=1}^7$ of $\frak{n}$ given by $\mu_1=\mbox{diag}\left(i, i, -2i\right)$, $\mu_2=\mbox{diag}\left(i, -i, 0\right),$
\begin{align*}
\mu_3=\left(\begin{smallmatrix}
0&-i&\\
-i&0&\\
&&0
\end{smallmatrix} \right), \mu_4=\left(\begin{smallmatrix}
&&1\\
&0&\\
-1&&
\end{smallmatrix} \right), \mu_5=\left(\begin{smallmatrix}
0&&\\
&0&-1\\
&1&0
\end{smallmatrix} \right), \mu_6=\left(\begin{smallmatrix}
&&-i\\
&0&\\
-i&&
\end{smallmatrix} \right),
\mu_7=\left(\begin{smallmatrix}
0&&\\
&0&-i\\
&-i&0
\end{smallmatrix} \right).      
\end{align*}
A straightfoward calculation yields
\begin{align*}
P_t=4\,\mbox{diag}\left(1,\cos^2t,\cos^2t,\sin^2(t/2),\sin^2(t/2),\cos^2(t/2),\cos^2(t/2)\right).
\end{align*}
By plugging $P_t$ into the equations  $\tau^{\tan}=0$ and $\tau^{\nor}=0$ we get
that $\tau^{\nor}=0$ is equivalent to the ODE and the identity $\tau^{\tan}=0$ is satisfied trivially.

\subsection{Initial value problem}
We prove that each solution $r$ of the BVP is smooth, i.e., we deal with the degenerate ends of the interval of definition.

\smallskip

In what follows we deal with the initial value problem at the left degenerate end of the interval $\left(0,\frac{\pi}{2}\right)$.
In order to solve this initial value problem we use a theorem of Malgrange in the version that can be found in \cite{haskins}. 

\smallskip

\noindent\textbf{Theorem of Malgrange (Theorem 4.7 in \cite{haskins}):}
\textit{Consider the singular initial value problem
\begin{align} 
\label{sing}
\dot y=\tfrac{1}{t}M_{-1}(y)+M(t,y),\hspace{1cm}y(0)=y_0,
\end{align}
where $y$ takes values in $\R^k$, $M_{-1}:\R^k\rightarrow\R^k$ is a smooth function of $y$ in a neighborhood of
$y_0$ and $M:\R\times\R^k\rightarrow\R^k$ is smooth in $t$, $y$ in a neighborhood of $(0,y_0)$. Assume that 
\begin{enumerate}
\renewcommand{\labelenumi}{(\roman{enumi})}
\item $M_{-1}(y_0) = 0$,
\item$h\mbox{Id}-d_{y_0}M_{-1}$ is invertible for all $h\in\N$, $h\geq1$.
\end{enumerate}
Then there exists a unique solution $y(t)$ of (\ref{sing}). Furthermore $y$ depends continuously on $y_0$ satisfying (i) and (ii).}

\smallskip

Next we finally solve the initial value problem at $t=0$.

\begin{theorem}
\label{ivp}
For each $v\in\R$ the initial value problem $r(t)_{\lvert t=0}=0, \dot r(0):=\tfrac{d}{dt}r(t)_{\lvert t=0}=v$
has a unique solution.
\end{theorem}
\begin{proof}
We introduce the variable $s=t^2$ and the operator $\theta=s\tfrac{d}{ds}$. Clearly, $\tfrac{d}{dt}=\tfrac{2}{\sqrt{s}}\theta$ and 
$\tfrac{d^2}{dt^2}=-\tfrac{2}{s}\theta+\tfrac{4}{s}\theta^2$. In terms of $s$ and $\theta$ the ODE is given by
\begin{align*}
\theta^2r=\tfrac{1}{2}\theta r-s\csc^2(2\sqrt{s})\left(\tfrac{\sin(4\sqrt{s})}{\sqrt{s}}\theta r+\sin^2\sqrt{s}\sin(2r)-8\cos^3\sqrt{s}\sin r\right)=:\psi
\end{align*}
Next we rewrite this ODE as a first order system
\begin{align*}
\theta (r)=\theta r,\hspace{1cm}\theta(\theta r)=\psi
\end{align*}
and we compute the partial derivatives of the right hand sides with respect to $r$ and $\theta r$ at $s=0$.
We thus obtain
\begin{align*}
\begin{pmatrix}
\tfrac{\partial}{\partial r}\theta r&\tfrac{\partial}{\partial \theta r}\theta r\\
\tfrac{\partial}{\partial r}\psi&\tfrac{\partial}{\partial \theta r}\psi
\end{pmatrix}_{\lvert s=0}=\begin{pmatrix}
0&1\\
\tfrac{1}{2}& -\tfrac{1}{2} 
\end{pmatrix}.
\end{align*}
Since the eigenvalues of this matrix are given by $\tfrac{1}{2}$ and $-1$, the Theorem of Malgrange states 
 that a formal
power series solution of this equation converges to a unique solution in a neighborhood of $s = 0$.
This solution depends continuously on $v$.
\end{proof}

Similarly we deal with the initial value problem at $t=\tfrac{\pi}{2}$.
All together we thus obtain the following theorem.

\begin{theorem}
Each solution of the BVP yields a harmonic self-map of $\SU(3)$.
\end{theorem}

\subsection{The variable $x$}
In terms of $x=\log\tan t$ the BVP transforms into
\begin{align*}
r''(x)-\tanh x\cdot r'(x)+\tfrac{1+\tanh x}{2}\sin2r(x)-\tfrac{1}{\sqrt{2}}(1-\tanh x)^{\frac{3}{2}}\sin r(x)=0,
\end{align*}
with $\lim_{x\rightarrow -\infty}r(x)=0$ and $\lim_{x\rightarrow \infty}r(x)=\frac{(2\ell+1)\pi}{2}$, $\ell\in\Z$. 
We thus have moved the endpoint of the interval of definition to $+\infty$ and $-\infty$, respectively.
This boundary value problem will henceforth also be referred to as BVP; it will become clear from the context whether we consider the variable $t$ or the variable $x$.


\subsection{Behavior of $r$ for positive $x$}
This subsection is structured as follows: after fixing some notation we introduce a Lyapunov function $W$ which turns out to be an important tool.
Afterwards we give a bound for the first derivative of each solution $r$ of the BVP. Finally, we give some restrictions for the solutions $r$ of the BVP, e.g.,
we prove that each solution of the ODE satisfies $\lim_{x\rightarrow\infty}r(x)=\ell \tfrac{\pi}{2}$ for a $\ell\in\Z$ or $\lim_{x\rightarrow\infty}r(x)=\pm\infty$.

\bigskip

\textit{Notation.} 
For the following considerations it is helpful to introduce the functions $f, g, i:\R\rightarrow\R$ and $h:\R^2\rightarrow\R$ by
\begin{align*}
f&:x\mapsto (1+\tanh x-\sqrt{2}(1-\tanh x)^{\frac{3}{2}})^{\frac{1}{2}},\\
g&:x\mapsto \coth x\left(\tfrac{1}{2}(1+\tanh x)+\tfrac{1}{\sqrt{2}}(1-\tanh x)^{\frac{3}{2}}\right),\\
h&:(x,r)\mapsto\tfrac{1+\tanh x}{2}\sin^2r-\sqrt{2}(1-\tanh x)^{\frac{3}{2}}\sin^2\tfrac{r}{2},\\
i&:x\mapsto (-1+\tanh x+2\sqrt{2}(1+\tanh x)^{\frac{3}{2}})^{\frac{1}{2}}.
\end{align*}

\bigskip

\textit{Lyapunov function.} Introduce $W:\R\rightarrow\R$ by
\begin{align*}
W(x)=\tfrac{1}{2}r'(x)^2+\tfrac{1+\tanh x}{2}\sin^2r(x)-\sqrt{2}(1-\tanh x)^{\frac{3}{2}}\sin^2\tfrac{r(x)}{2},
\end{align*}
which turns out to be a Lyapunov function.

\begin{lemma}
\label{increase2}
Either the function $W$ is strictly increasing for $x\geq 0$ or $W\equiv 0$. Furthermore, $W\equiv 0$ if and only if $r\equiv 2k\pi$ for a $k\in\Z$.
\end{lemma}

\begin{proof}
Using the ODE we obtain 
\begin{align*}
\tfrac{d}{dx}W(x)=\tanh x\cdot r'(x)^2+\mbox{sech}^2x\left(\tfrac{1}{2}\,\sin^2r(x)+\tfrac{3}{\sqrt{2}}(1-\tanh x)^{\frac{1}{2}}\sin^2\tfrac{r(x)}{2}\right)\geq 0
\end{align*}
for all $x\geq 0.$ Either $\tfrac{d}{dx}W(x)>0$ for all $x>0$ and then $W$ increases strictly or there exists a
$x_0>0$ such that $\tfrac{d}{dx}W(x_0)=0$. 
Thus $r'(x_0)=0$ and $r(x_0)=2k\pi$ for an $k\in\Z$. Hence the theorem of Picard-Lindel\"of yields $r\equiv 2k\pi$ and therefore $W\equiv 0$.
\end{proof}

\bigskip

\textit{Bounds for the first derivative of $r$.}
In the next lemma we prove that for each solution $r$ of the BVP the first derivative is bounded by a constant.

\begin{lemma}
\label{bounded2}
If $W(x_0)>1$ for one $x_0\geq 0$ then $\lim_{x\rightarrow\infty}r(x)=\pm\infty$.
In particular, if $\lvert r'(x_0)\lvert >(2(1+\sqrt{2}))^{\frac{1}{2}}$ for a point $x_0\geq 0$ then $\lim_{x\rightarrow\infty}r(x)=\pm\infty$.
\end{lemma}
\begin{proof}
If $W(x_0)>1$ for an $x_0\geq 0$ then Lemma\,\ref{increase2} implies
$W(x)\geq W(x_0)>1$ for all $x\geq x_0$. Since $W(x)=\tfrac{1}{2}r'(x)^2+h(x,r(x))$ we have
$$r'(x)^2 \geq 2W(x_0)-2h(x,r(x))\geq 2W(x_0)-2>0$$
for $x\geq x_0$. This establishes the first claim. Since $\lvert r'(x_0)\lvert >(2(1+\sqrt{2}))^{\frac{1}{2}}$ implies $W(x_0)>1$, the 
second claim is an immediate consequence of this.
\end{proof}

In the next lemma we improve the result of the previous lemma for those $x\geq 0$ for which $g(x)<(2(1+\sqrt{2}))^{\frac{1}{2}}$.

\begin{lemma}
\label{bound}
If $\lvert r'(x_0) \rvert>g(x_0)$ for an $x_0>0$ then $\lim_{x\rightarrow\infty}\pm r(x)=\infty$.
\end{lemma}
\begin{proof}
We can assume without loss of generality $r'(x_0)>g(x_0)$ for an $x_0>0$.
If $-r'(x_0)>g(x_0)$ for an $x_0>0$, we consider $-r$ instead of $r$.
Consequently,
\begin{align*}
r'(x_0)> g(x_0)\geq \coth x_0\left(\tfrac{1}{2}(1+\tanh x_0)\sin 2r(x_0)-\tfrac{1}{\sqrt{2}}(1-\tanh x_0)^{\frac{3}{2}}\sin r(x_0)\right).
\end{align*}
Since for $x>0$ the inequality $r''(x)>0$  is equivalent to
\begin{align*}
r'(x)>\coth x\left(\tfrac{1}{2}(1+\tanh x)\sin2r(x)-\tfrac{1}{\sqrt{2}}(1-\tanh x)^{\frac{3}{2}}\sin r(x)\right),
\end{align*}
we get $r''(x_0)>0$. Assume that there exists a point $x_1>x_0$ such that $r''(x)>0$ for all $x\in\left[x_0,x_1\right)$
and $r''(x_1)=0$. Since $g$ decreases on the positive $x$-axis, we get 
$r'(x)\geq r'(x_0)> g(x_0)\geq g(x)$
for $x\in\lbrack x_0,x_1\rbrack$. Therefore
\begin{align*}
r'(x_1)>g(x_1)\geq\coth x_1\big(\tfrac{1}{2}(1+\tanh x_1)\sin2r(x_1)-\tfrac{1}{\sqrt{2}}(1-\tanh x_1)^{\frac{3}{2}}\sin r(x_1)\big).
\end{align*}
Hence $r''(x_1)>0$, which contradicts our assumption.
Consequently, we have $r''(x)>0$ for all $x\geq x_0$ and thus $r'(x)\geq r'(x_0)>0$ for $x\geq x_0$. Hence $\lim_{x\rightarrow\infty}r(x)=\infty$, which establishes the claim.

\smallskip

The second claim follows from the first by considering $-r$ instead of $r$.
\end{proof}

\bigskip 

\textit{Restrictions for $r$.}
Let $d^{+}>0$ be the unique positive solution of $f(x)=g(x)$.
It is straightforward to verify that $f$ increases strictly on the positive $x$-axis, while $g$ decreases strictly in this domain.
Hence $f(x)\geq g(x)$ for all $x\geq d^{+}$.

\smallskip

The next lemma states that the graph of each solution of the BVP has to be contained in a stripe of height $3\pi$. 

\begin{lemma}
\label{dp} 
\begin{compactenum}[(i)]
\item If there exists a point $x_0\geq d^{+}$ with $r(x_0)=(4k+1)\tfrac{\pi}{2}$, $k\in\Z$, and $r'(x_0)\geq 0$ then
$\lim_{x\rightarrow\infty}r(x)=\infty$.
\item If there exists a point $x_0\geq d^{+}$ with $r(x_0)=(4k+3)\tfrac{\pi}{2}$, $k\in\Z$, and $r'(x_0)\leq 0$ then
$\lim_{x\rightarrow\infty}r(x)=-\infty$.
\end{compactenum}
\end{lemma}

\begin{proof}
It is sufficient to prove the first statement:
if $r(x_0)=(4k+3)\frac{\pi}{2}$, $k\in\Z$, and $r'(x_0)\leq 0$ for an $x_0\geq d^{+}$ then $-r(x_0)=-(4k+3)\frac{\pi}{2}=(4(-k-1)+1)\frac{\pi}{2}$ and $-r'(x_0)\geq 0$.
Applying the first result to $-r$ thus yields the second statement.

\smallskip

Assume that there exists a point $x_0\geq 0$ with $r(x_0)=(4k+1)\frac{\pi}{2}$, $k\in\Z$, and $r'(x_0)\geq 0$.
If $r$
is a solution of the ODE, so are $r+2\pi j$, $j\in\Z$. Consequently, we may assume without loss of generality that $k=0$. 
Since $r(x_0)=\tfrac{\pi}{2}$ and $r'(x_0)\geq 0$ the ODE implies $r''(x_0)>0$. Consequently, there exists a point $x_2>x_0$ such that $\frac{\pi}{2}<r(x_2)<\pi$ and $r'(x_2)>0$.
The ODE thus implies the existence of a point $x_1>x_0$ with $r(x_1)=\pi$ and $r'(x_1)\geq 0$. Since $r'(x_1)=0$ would imply $r\equiv \pi$ we have $r'(x_1)>0$. Thus by Lemma\,\ref{increase2} we have $W(x_1)\geq W(x_0)$. This in turn implies
\begin{align*}
r'(x_1)^2\geq (1+\tanh x_0)+2^{3/2}(1-\tanh x_1)^{\frac{3}{2}}-\sqrt{2}(1-\tanh x_0)^{\frac{3}{2}}\geq f(x_0)^2.
\end{align*}
Since $r'(x_1)>0$ we get $r'(x_1)\geq f(x_0)\geq g(x_0)>g(x_1)$ and thus
Lemma\,\ref{bound} establishes the claim.
\end{proof}

In the following lemma we prove that $\lim_{x\rightarrow\infty}r(x)$ can only attain certain values.

\begin{lemma}
\label{limit3}
Either $\lim_{x\rightarrow\infty}r(x)=k\tfrac{\pi}{2}$ for an $k\in\Z$ or $\lim_{x\rightarrow\infty} r(x)=\pm\infty$.
\end{lemma}

\begin{proof}
If $r$ is constant then the ODE implies $r\equiv j\pi$ for a $j\in\Z$ and thus $\lim_{x\rightarrow\infty}r(x)=j\pi$.
Therefore we may assume that $r$ is non-constant.
Hence $W$ increases strictly by Lemma\,\ref{increase2}. In particular $\lim_{x\rightarrow\infty}W(x)$ exists, where this limit might possibly be $\infty$.

\smallskip

Let us first assume that $\lim_{x\rightarrow\infty} r'(x)=0$.
Then $\lim_{x\rightarrow\infty}W(x)=\lim_{x\rightarrow\infty}\sin^2r(x)$ exists, which in turn implies that $\lim_{x\rightarrow\infty}r(x)$ exists and is finite. Thus the ODE yields
$\lim_{x\rightarrow\infty}r''(x)=-\sin(2\lim_{x\rightarrow\infty}r(x))$.
Consequently, $\lim_{x\rightarrow\infty}r(x)=k\tfrac{\pi}{2}$ for an $k\in\Z$ since otherwise we would obtain a contradiction to
the assumption $\lim_{x\rightarrow\infty} r'(x)=0$.

\smallskip

Next we assume $\lim_{x\rightarrow\infty} r'(x)\neq 0$, which implies $\lim_{x\rightarrow\infty}\tfrac{d}{dx}W(x)\neq 0$. 
Since $W$ increases strictly, we get $\lim_{x\rightarrow\infty}W(x)=\infty$. This in turn implies $\lim_{x\rightarrow\infty} r'(x)^2=\infty$.
Thus for every $\epsilon>0$ there exists a point $x_0\in\R$ such that $\lvert r'(x)\lvert > \epsilon$ for all $x>x_0$.
Consequently, $\lim_{x\rightarrow\infty} r(x)=\pm\infty$.
\end{proof}

The next lemma should be considered as completion of Lemma\,\ref{dp}:
 we deal with the cases where there exists a point $x_0\geq d^{+}$ with\\\hspace{1cm} (1)
$r(x_0)=(4k+1)\tfrac{\pi}{2}$, $k\in\Z$, and $r'(x_0)<0$;\\\hspace{1cm} 
(2) $r(x_0)=(4k+3)\tfrac{\pi}{2}$, $k\in\Z$, and $r'(x_0)>0$.

\begin{lemma}
\label{limit2}
The following two statements hold:
\begin{compactenum}[(1)]
\item If there exists a point $x_0\geq d^{+}$ with $r(x_0)=(4k+1)\tfrac{\pi}{2}$, $k\in\Z$, and $r'(x_0)<0$ then
$\lim_{x\rightarrow\infty}r(x)=\pm\infty$ or $\lim_{x\rightarrow\infty}r(x)=(4k+1)\tfrac{\pi}{2}$. 
\item If there exists a point $x_0\geq d^{+}$ with $r(x_0)=(4k+3)\tfrac{\pi}{2}$, $k\in\Z$, and $r'(x_0)>0$ then
$\lim_{x\rightarrow\infty}r(x)=\pm\infty$ or $\lim_{x\rightarrow\infty}r(x)=(4k+3)\tfrac{\pi}{2}$.
\end{compactenum}
\end{lemma}

\begin{proof}
As in the proof of Lemma\,\ref{dp} one sees that the first claim implies the second claim.

\smallskip

In what follows we assume that there exists a point $x_0\geq d^{+}$ with $r(x_0)=(4k+1)\tfrac{\pi}{2}$, $k\in\Z$, and $r'(x_0)<0$.
If $r$ solves the ODE, so does $r+2\pi j$, $j\in\Z$. Thus we may assume $k=0$, i.e., $r(x_0)=\tfrac{\pi}{2}$.
Then either of the following three cases occurs:
\begin{compactenum}[(1)]
\renewcommand{\labelenumi}{(\roman{enumi})}
\item there exists an $x_1>x_0$ such that $r(x_1)=\tfrac{\pi}{2}$ and $r'(x_1)\geq0$,
\item there exists an $x_2>x_0$ such that $r(x_2)=0$ and $r'(x_2)\leq 0$.
\item we have $0<r(x)<\tfrac{\pi}{2}$ for all $x>x_0$,
\end{compactenum}

If the first case arises, then Lemma\,\ref{dp} implies 
$\lim_{x\rightarrow\infty}r(x)=\infty$. 
Next assume that the second case occurs.
Since $W$ increases strictly we get $W(x_0)<W(x_2)$ which implies $-r'(x_2)\geq f(x_0)>g(x_0)>g(x_2)$.
Thus Lemma\,\ref{bound} implies $\lim_{x\rightarrow\infty}r(x)=-\infty$. 

\smallskip

Finally, we deal with the third case. By Lemma\,\ref{limit3} we have $\lim_{x\rightarrow\infty}r(x)=\tfrac{\pi}{2}$ or $\lim_{x\rightarrow\infty}r(x)=0$. 
The latter case cannot occur:
from $x_0\geq d^{+}$ and $r(x_0)=\tfrac{\pi}{2}$ we deduce $h(x_0,r(x_0))>0$ and thus
$W(x_0)>0$. Consequently, Lemma\,\ref{increase2} implies that $\lim_{x\rightarrow\infty}r(x)=0$ is not possible and thus we have $\lim_{x\rightarrow\infty}r(x)=\tfrac{\pi}{2}$.

\smallskip

The second statement of the lemma is proved analogously.
\end{proof}

Using the preceding lemma we show that each solution of the ODE with $\lim_{x\rightarrow\infty}r(x)=k\pi$ oscillates infinitely many times around $k\pi$.
This result allows us later on to show that none of the constructed solutions $r$ of the BVP can satisfy $\lim_{x\rightarrow\infty}r(x)=k\pi$.

\begin{lemma}
\label{streifen}
If  $\lim_{x\rightarrow\infty}r(x)=k\pi$ for an $k\in\Z$ then $r$ oscillates infinitely many times around $k\pi$.
\end{lemma}
\begin{proof}
By Lemma\,\ref{dp} and Lemma\,\ref{limit2} we have $(2k-1)\tfrac{\pi}{2}<r(x)<(2k+1)\tfrac{\pi}{2}$ for all $x\geq d^{+}$.
If $r$ is a solution of the ODE, so are the functions $r+2\pi j$, $j\in\Z$.
Hence we may assume without loss of generality that $k\in\lbrace 0,1\rbrace$.

\smallskip

Let us first consider $k=0$, i.e. we have $\lim_{x\rightarrow\infty}r(x)=0$ by assumption.\\ 
We start by proving that $r$ cannot converge against $0$ \lq from above\rq, i.e. there cannot exist an $x_0>0$
such that $r(x)\geq 0$ for all $x\geq x_0$ and $\lim_{x\rightarrow\infty}r(x)=0$.\\
We prove this by contradiction. Let $x_1>0$ such that $-\tfrac{1+\tanh x}{2}+\tfrac{1}{\sqrt{2}}(1-\tanh x)^{\frac{3}{2}}<0$ for all $x>x_1$.
By asumption there exists a $x_2>x_1$ such that $0<r(x_2)<\tfrac{\pi}{3}$ and $r'(x_2)<0$. The ODE thus implies
\begin{align*}
r''(x)&=\tanh x\, r'(x)+\left(\tfrac{1}{\sqrt{2}}(1-\tanh x)^{\frac{3}{2}}-(1+\tanh x)\cos r(x)\right)\sin r(x)\\
&\leq\tanh x\, r'(x)+\left(\tfrac{1}{\sqrt{2}}(1-\tanh x)^{\frac{3}{2}}-\tfrac{(1+\tanh x)}{2}\right)\sin r(x)<0
\end{align*}
for all $x\geq x_2$ for which $0<r(x)<\tfrac{\pi}{3}$.
Consequently, there exists an $x_3>x_2$ such that $r(x_3)=0$ and $r'(x_3)<0$.
Hence there exists a point $x_4>x_3$ with $r(x_4)<0$, which contradicts our assumption.\\
Similarly, we prove that $r$ cannot converge against $0$ \lq from below\rq, i.e. there cannot exist an $x_0>0$
such that $r(x)\leq 0$ for all $x\geq x_0$ and $\lim_{x\rightarrow\infty}r(x)=0$.
More precisely, we show that if there exists a $x_5>x_1$ such that $-\tfrac{\pi}{3}<r(x_5)<0$ and $r'(x_5)>0$ then the ODE implies that there exists an $x_6>x_5$ such that $r(x_6)=0$ and $r'(x_6)>0$.

\smallskip

Since we have $\lim_{x\rightarrow\infty}r(x)=0$ by assumption, the above considerations imply that $r$ oscillates infinitely many times around $0$. The case $k=1$ is treated similarly.
\end{proof}

\subsection{Behavior of $r$ for negative $x$}
In this subsection we prove that there exist a $d^-<0$  such that for each solution $r$ of the ODE with $\lim_{x\rightarrow -\infty}r(x)=0$ we have $-2\pi<r(x)<2\pi$ for all $x<d^-$.
The proofs of those results which are proved in analogy to the corresponding results of the preceding subsection are omitted.

\smallskip

In terms of $\phi(x)=r(-x)-\tfrac{3\pi}{2}$ the ODE transforms into
\begin{align}
\label{odez}
\phi''(x)-\tanh x\cdot\phi'(x)-\tfrac{1-\tanh x}{2}\sin2\phi(x)+\tfrac{1}{\sqrt{2}}(1+\tanh x)^{\frac{3}{2}}\cos\phi(x)=0.
\end{align}

For any solution $\phi$ of the ODE (\ref{odez}) introduce the function $W^{\phi}:\R\rightarrow\R$ by
\begin{align*}
W^{\phi}(x)=\tfrac{1}{2}\phi'(x)^2-\tfrac{1-\tanh x}{2}\sin^2\phi(x)+\sqrt{2}(1+\tanh x)^{\frac{3}{2}}\sin^2(\tfrac{1}{2}\phi(x)-\tfrac{3\pi}{4}),
\end{align*}
which turns out to be a Lyapunov function.

\begin{lemma}
\label{bound22}
The function $W^{\phi}$ increases strictly on the non-negative $x$-axis.
For any solution $\phi$ of the ODE (\ref{odez}) with $\lim_{x\rightarrow\infty}\phi(x)=-\tfrac{3\pi}{2}$ 
we have $\lvert \phi'(x)\lvert\leq 3$ for $x\geq 0$.
\end{lemma}

\begin{lemma}
\label{bound2}
Let $\phi$ be a solution of the ODE (\ref{odez}).
If there exists a point $x_0>0$ with $\lvert\phi'(x_0)\lvert>-g(-x_0)$ then $\lim_{x\rightarrow\infty}\pm\phi(x)=\infty$.
\end{lemma}

Let $d_{-}>0$ be the unique positive solution of the equation $i(x)+g(-x)=0$.
Note that we have $i(x)\geq -g(-x)$ for all $x\geq d_{-}$. Set $d^{-}:=-d_{-}$.

\begin{lemma}
\label{dm}
For each solution $r$ of the ODE with $\lim_{x\rightarrow -\infty}r(x)=0$ we have $-2\pi<r(x)<2\pi$ for all $x<d^-$.
\end{lemma}
\begin{proof}
Let $\phi$ solve the ODE (\ref{odez}) with $\lim_{x\rightarrow\infty}\phi(x)=-\frac{3\pi}{2}$. We prove that
there cannot exists a point $x_0\geq d_{-}$ with $\phi(x_0)=\frac{\pi}{2}$ or $\phi(x_0)=-\frac{7\pi}{2}$.
This statement is obviously equivalent to the claim.

\smallskip

Suppose $\phi(x_0)=\tfrac{\pi}{2}$ for an $x_0\geq d_{-}$. Since $\lim_{x\rightarrow\infty}\phi(x)=-\tfrac{3\pi}{2}$, continuity of $\phi$ implies that there exists a point $x_1>x_0$ with
$\phi(x_1)=-\tfrac{\pi}{2}$. Thus $W^{\phi}(x_1)\geq W^{\phi}(x_0)$,
which is equivalent to
\begin{align*}
\tfrac{1}{2}\phi'(x_1)^2-\tfrac{1}{2}(1-\tanh x_1)\geq 
\tfrac{1}{2}\phi'(x_0)^2-\tfrac{1}{2}(1-\tanh x_0)+\sqrt{2}(1+\tanh x_0)^{\frac{3}{2}}.
\end{align*}
This in turn implies $\phi'(x_1)^2\geq i(x_0)^2$.
Consequently, we either have $\phi'(x_1)\geq i(x_0)\geq -g(-x_0)\geq -g(-x_1)$
or $\phi'(x_1)\leq -i(x_0)\leq g(-x_0)\leq g(-x_1)$.
Lemma\,\ref{bound2} thus yields $\lim_{x\rightarrow\infty}\phi(x)=\pm\infty$, which contradicts our assumption.

\smallskip

The case $\phi(x_0)=-\tfrac{7\pi}{2}$ for an $x_0\geq d_{-}$ is treated analogously.
\end{proof}

\subsection{Restrictions on the Brouwer degree}
\label{sec5}
In this subsection we prove Theorem\,A:
by combining the results of the previous subsections we give restrictions for the possible integers $\ell$ in $r(\tfrac{\pi}{2})=(2\ell+1)\tfrac{\pi}{2}$.
Theorem\,3.4 in \cite{puttmann} implies that the Brouwer degree of $\psi_r$ is given by
$\mbox{deg}\,\psi_{r}=2\ell+1$. 
Consequently, giving a restriction for the possible $\ell$ is equivalent to giving a restriction for the possible Brouwer degrees.


\begin{theorem}
\label{brodeg}
Each solution $r$ of the BVP has Brouwer degree $\pm 1, \pm 3$, $\pm 5$ or $\pm 7$.
\end{theorem}

\begin{proof}
The strategy for the proof is as follows.
\begin{enumerate}
\item By Lemma\,\ref{dm} there exists a constant $d^{-}<0$ such that all solutions $r$ of the BVP satisfy $-2\pi\leq r(x)\leq 2\pi$ for all $x\leq d^{-}$.
\item We prove that each solution $r$ of the BVP satisfies $-7\tfrac{\pi}{2}<r(d^+)<7\tfrac{\pi}{2}$.  In order to do so we use (1) and the fact that for each solution $r$ of the BVP the first derivative of $r$ is bounded by a constant 
\item Lemmas\,\ref{dp} and \ref{limit2} imply that if there exists $k\in\Z$ such that $(2k-1)\tfrac{\pi}{2}\leq r(d^+)\leq (2k+1)\tfrac{\pi}{2}$, then
either $\lim_{x\rightarrow\infty}r(x)=(2k\pm 1)\tfrac{\pi}{2}$ or $\lim_{x\rightarrow\infty}r(x)=\pm\infty$.
Since by assumption $r$ is a solution of the BVP the latter case does not occur. By (2) and Theorem\,3.4 in \cite{puttmann} we thus have
$\mbox{deg}(\psi_r)\in\lbrace\pm 1, \pm 3, \pm 5, \pm 7\rbrace$.

\end{enumerate}

\smallskip

It remains to prove (2).
Let us first consider the region $d^-\leq x\leq 0$. By Lemma\,\ref{bound22} and Lemma\,\ref{bound2} we have $\lvert r'(x)\lvert \leq 3$ and $r'(x)\leq -g(x)$ for all $d^-\leq x\leq 0$.
Let $x_0<0$ be such that $-g(x_0)=3$. Then we have 
\begin{align*}
r(0)\leq\int_{d^-}^{x_0}-g(x)dx-3x_0+r(d^-).
\end{align*}
Let us next consider the region $0\leq x\leq d^+$.
By Lemma\,\ref{bounded2} and Lemma\,\ref{bound} we have $\lvert r'(x)\lvert \leq (2(1+\sqrt{2}))^{\frac{1}{2}}$ and $r'(x)\leq g(x)$ for all $0\leq x\leq d^+$.
Let $x_1>0$ be such that $g(x_1)=(2(1+\sqrt{2}))^{\frac{1}{2}}$. Then we have 
\begin{align*}
r(d^+)\leq (2(1+\sqrt{2}))^{\frac{1}{2}}x_1+\int_{x_1}^{d^+}g(x)dx+r(0).
\end{align*}
By the above estimate for $r(0)$ in this inequality and using the computer program Mathematica to evaluate the integrals, we thus obtain
$r(d+)<\tfrac{3\pi}{2}+r(d^-)<\tfrac{7\pi}{2}$. Analogously, we prove $r(d+)>-\tfrac{3\pi}{2}+r(d^-)>-\tfrac{7\pi}{2}$, which establishes (2) and thus the claim.
\end{proof}

The preceding result does not seem to be optimal: numerical results indicate that all solutions of the BVP have Brouwer degree $\pm 1$ or $\pm 3$.
So the following question remains.

\begin{question*}
Do all solutions of the BVP have Brouwer degree $\pm 1$ or $\pm 3$?
\end{question*}


\section{Construction of infinitely many harmonic self-maps of $\SU(3)$}
\label{sec3}
First of all, note that in terms of the variable $x$ Theorem\,\ref{ivp} states that for every $v\geq 0$ there is a unique solution $r_v:\R\rightarrow\R$ of the ODE
that satisfies $r_v(x)\simeq v\exp(x)$ for $x\rightarrow -\infty$.
The functions $r_v$ and $r_v$ depend continuously on $v$.

\smallskip

We introduce the \textit{nodal number $\frak{N}(r_v)$ of $r_v$} as the number 
of intersection points of $r_v$ with $\pi$. The function $r_1(x)=\arctan\exp x$, i.e. $r(t)=t$, solves the BVP with $\frak{N}(r_1)=0$.
The next lemma ensures that we cannot increase $v$ arbitrarily without increasing the nodal number of $r_v$.
Its proof is based on Gastel's ideas, see Lemmas\,3.3 and 4.2 in \cite{ga}.

\begin{lemma}
\label{nullstellen}
For each $k\in\N$ there exists $c(k)>0$ such that $\frak{N}(r_v)\geq k$ for $v>c(k)$.
\end{lemma}
\begin{proof}
We denote by $\psi:\R\rightarrow\R$ the solution of the differential equation
\begin{align*}
\tfrac{d^2}{dx^2}\psi(x)+\tfrac{d}{dx}\psi(x)+2\sin\psi(x)=0,
\end{align*}
satisfying $\psi(x)\simeq -\pi+\exp(vx)$ as $x\rightarrow -\infty$.
There exists a unique solution with this properties, which can be proved as in \cite{ga}.
We define
$U(\psi,x):=\psi'(x)^2 +8\sin^2\tfrac{\psi(x)}{2}$, where we make use of the abbreviation $\psi'(x):=\tfrac{d}{dx}\psi(x)$.
By using the above differential equation we thus obtain
\begin{align}
\label{mon}
\tfrac{d}{dx}U(\psi, x) = -2\psi'(x)^2.
\end{align}
Consequently, $U(\psi,\,\cdot\,)$ is monotonically decreasing. Since it is also bounded from below by $0$, its limit for $x\rightarrow\infty$ exists, which can be only $0$
 by the above ordinary differential equation and (\ref{mon}), i.e., we have $U(\psi, \infty)=0$.
 
\smallskip
  
A solution $\psi$ of the above differential equation converges to $0$ as $x\rightarrow\infty$, and so does $\psi'$, because of (\ref{mon}). 
From this and the fact that $\psi$ asymptotically solves 
\begin{align*}
\tfrac{d^2}{dx^2}\psi(x)+\tfrac{d}{dx}\psi(x)+2\psi(x)=0,
\end{align*}
we get
\begin{align*}
\psi(x)\simeq c_1\exp(-x/2)\sin(\omega x-c_2)
\end{align*}
as $x\rightarrow\infty$, with constants $c_1,c_2\in\R$ and $\omega=\tfrac{1}{2}\sqrt{7}$.
As $v\rightarrow\infty$, the functions
$$\varphi_v:=r_v-\pi$$ converge to $\psi$ in $C^1(\R)$, which is proved as in Lemma\,3.3 in \cite{ga}.
This in turn implies the claim.
 \end{proof}

We now prove Theorem\,C: we show that for each $k\in\N$ there exist
a solution of the ODE with nodal number $k$. 



\begin{theorem}
\label{infam}
For each $k\in\N_0$ there exists a solution $r_v$ of the BVP with $\frak{N}(r_v)=k$.
Infinitely many of these solutions have Brouwer degree of absolute value greater or equal to three.
\end{theorem}
\begin{proof}
The strategy of the proof is to show that for each $k\in\N$ the function $r_{v_k}$, with $v_k=\mbox{sup}\left\{v\,\lvert\,\frak{N}(r_{v})=k\right\}$, is a solution of the BVP with nodal number $k$. 

\smallskip

\textit{First Step: consider $r_{v_0}$.}\\
The function $r_1(x)=\arctan\exp x$ solves the BVP with $\frak{N}(r_1)=0$.
Consequently, $v_0=\mbox{sup}\left\{v\,\lvert\,\frak{N}(r_v)=0\right\}$ is well-defined and
Lemma\,\ref{nullstellen} implies $v_0<\infty$.\\ 
We prove $\frak{N}(r_{v_0})=0$ by contradiction, i.e., we
assume that there exists a point $x_0\in\R$ with $r_{v_0}(x_0)=\pi$.
We have $r_{v_0}'(x_0)\neq 0$ since otherwise $r\equiv \pi$ which contradicts our assumption. Consequently, $r_{v_0}-\pi$ has opposite signs in the intervals $(-\infty,x_0)$ and $(x_0,\infty)$, respectively.
Since $r_v$ depends continuously on $v$ there exists a sequence $(c_i)_{i\in\N}$ with
$c_i<v_0$, $\lim_{i\rightarrow\infty}c_i=v_0$ and $\frak{N}(r_{c_i})=0$. Thus each of the functions $r_{c_i}-\pi$ has a different sign than $r_{v_0}-\pi$ on the interval $(x_0,\infty)$. This contradicts the fact that $r_v-\pi$ depends continuously on $v$.
Consequently, $\frak{N}(r_{v_0})=0$.

\smallskip

\textit{Second Step: there exists $\epsilon>0$ such that $\frak{N}(r_v)=1$ for $v\in \left(v_0,v_0+\epsilon\right)$.}\\
Recall that there cannot exist a point $x_0\in\R$ such that $r_v(x_0)=\pi$ and $r'_v(x_0)=0$. 
Since $r_v$ depends continuously on $v$, an additional node can thus only arise at infinity, i.e., there exists $\epsilon>0$ such that
$r_v-\pi$ has at least one zero $z_1(v)$ for each $v\in \left(v_0,v_0+\epsilon\right)$ and $\lim_{v\searrow v_0}z_{1}(v)=\infty$.  
Clearly, $r_v'(z_1(v))>0$.

\smallskip

Let $v\in \left(v_0,v_0+\epsilon\right)$ and $x_1\in\R$ such that $\tanh(x)\geq 0.99$ for all $x\geq x_1$.
Lemma\,\ref{limit3} implies that $\lim_{x\rightarrow\infty}r_{v_0}(x)=j\tfrac{\pi}{2}$ with $j\in\Z, j\leq 1$ or $\lim_{x\rightarrow\infty}r_{v_0}(x)=-\infty.$
Consequently, for $\delta\in (0,\tfrac{\pi}{2})$ we can choose the above $\epsilon$ so small that $\mbox{min}\,\lbrace \lvert r_v(x)-\tfrac{\pi}{2}\lvert\,\,\mid {x\geq x_1}\rbrace\leq\delta$.
Note that the minimum exists since we can minimize over the compact interval $\lbrack x_1,z_1(v)\rbrack$.
In other words, we can choose $\epsilon$ so small that $r_v$ becomes arbitrary close to $\tfrac{\pi}{2}$ on the interval $\lbrack x_1,z_1(v)\rbrack$.
Let $x(v)\geq 0$ be such that $\lvert r_v(x(v))-\tfrac{\pi}{2}\lvert\leq\delta$. Clearly, $x(v)<z_1(v)$.

\smallskip

By Lemma\,\ref{increase2} we have $W(z_1(v))>W(x(v))$ which implies $$r_v'(z_1(v))^2>2W(x(v)).$$
We choose $\delta>0$ so small that
\begin{align*}
r_v'(z_1(v))\geq 1.1.
\end{align*}
By the ODE we thus get $r_v''(x)>0$ for all $x\geq z_1(v)$, i.e., these solution all have nodal number equal to one.
Consequently, we have shown that there exists an $\epsilon>0$ such that $\frak{N}(r_v)=1$ for $v\in \left(v_0,v_0+\epsilon\right)$.
Furthermore, $v_1=\mbox{sup}\left\{v\,\lvert\,\frak{N}(r_{v})=1\right\}$ is well-defined and $v_1>v_0$. Lemma\,\ref{nullstellen} implies $v_1<\infty$. 

\smallskip

\textit{Third Step: proceed inductively.}\\
Lemma\,\ref{nullstellen} implies that for each $k\in\N$ the number $v_k=\mbox{sup}\left\{v\,\lvert\,\frak{N}(r_{v})=k\right\}$ is well-defined and in particular finite.
Furthermore, as in Step $2$ we prove $v_k>v_{k-1}$. Analogously to the considerations for $v_1$ we prove that $\varphi_{v_k}$ has exactly $k$ zeros and that there exists $\epsilon_k>0$ such that each $\varphi_v$, $v\in\left(v_{k},v_{k}+\epsilon_{k}\right)$,
has exactly $k+1$ zeros.

\smallskip

\textit{Fourth Step: for each $k\in\N_0$, $r_{v_k}$ is a solution of the BVP.}\\
Since $\frak{N}(r_{v_k})=k$, Lemma\,\ref{streifen} implies that $\lim_{x\rightarrow\infty}r_k(x)=j\pi$ for an $j\in\Z$ is not possible.
Consequently, Lemma\,\ref{limit3} implies that there exists $\ell_0\in\Z$ such that $\lim_{x\rightarrow \infty}r_{v_i}(x)=\ell_0\pi+\frac{\pi}{2}$ or $\lim_{x\rightarrow \infty}r_{v_i}(x)=\pm\infty$.

\smallskip

Below we assume that the later case occurs. We may assume without loss of generality $\lim_{x\rightarrow \infty}r_{v_k}(x)=-\infty$.
Recall $\frak{N}(r_{v_k})=k$ and that there exists an $\epsilon_k>0$ such that $\frak{N}(r_{v})=k+1$ for $v\in\left(v_k,v_k+\epsilon_k\right)$.
Similarly as in Step $2$ we prove that we can choose $\epsilon_k>0$ such that $\lim_{x\rightarrow \infty}\varphi_{v}(x)=\infty$
for $v\in\left(v_{k},v_{k}+{\epsilon_{k}}\right)$.

\smallskip

On the other hand, the fact that $\varphi_{v}$ depends continuously on $v$ implies that
for each $v\in\left(v_{k},v_{k}+{\epsilon_{k}}\right)$ there exist $k_0\in\Z$ and $x_{k_0}>d^{+}$ such that $\varphi_v(x_{k_0})=(4k_0+3)\tfrac{\pi}{2}$ and
$\varphi_v'(x_{k_0})<0$. Lemma\,\ref{dp} thus implies $\lim_{x\rightarrow -\infty}\varphi_{v}(x)=-\infty$, which contradicts the results of the preceding paragraph.
Consequently, there exists $\ell_0\in\Z$ such that $\lim_{x\rightarrow \infty}r_{v_k}(x)=\ell_0\pi+\tfrac{\pi}{2}$ and thus each $r_{v_k}$, $k\in\N$, is a solution of the BVP with nodal number $k$.
This proves the first claim.

\smallskip

The second claim is an immediate consequence of the above construction and Theorem\,3.4 in \cite{puttmann}.
\end{proof}

\section{Limit configuration}
\label{sec4}
After providing one preparatory lemma we show Theorem\,D. 

\smallskip

For any solution $r$ of the ODE we define the function $W_{-}:\R\rightarrow\R$ by
\begin{align*}
W_{-}(x)=\tfrac{1}{2}r'(x)^2-\tfrac{1+\tanh x}{2}\cos^2(r(x))+\sqrt{2}(1-\tanh x)^{\tfrac{3}{2}}\cos^2(\tfrac{1}{2}r(x)).
\end{align*}
The function $W_{-}$ decreases strictly for $x\leq 0$.
First we show that for every interval of the form $\lbrack x_0, d^-\rbrack$, the energy $W_{-}$ of $r_v$ becomes arbitrarily small on this interval
if we chose the velocity $v$ to be \lq large enough\rq. Keep in mind that the energy $W_{-}$ of $r_v$ depends on $v$.

\begin{lemma}
\label{bigv}
For $\epsilon>0$ and $x_0\leq d^{-}$ there exists $v_0>0$ such that the energy $W_{-}$ of $r_v$ satisfies $W_{-}(x)<\epsilon$ for $x_0\leq x\leq d^{-}$ and $v\geq v_0$.
\end{lemma}
\begin{proof}
Since $\lim_{x\rightarrow -\infty}r_v'(x)=0$, there exists $x_1\leq d^{-}$ such that $r_v'(x)^2<\epsilon$ for $x\leq x_1$.
Furthermore, by the proof of Lemma\,3.3 in \cite{ga} we get
\begin{align}
\label{limit}
\lim_{v\rightarrow\infty}\varphi_v(x-\log v)=\psi(x)
\end{align}
for all $x\in\R$, where
 $\psi:\R\rightarrow\R$ is the unique solution of the differential equation
\begin{align*}
\tfrac{d^2}{dx^2}\psi(x)+\tfrac{d}{dx}\psi(x)+2\sin\psi(x)=0,
\end{align*}
satisfying $\psi(x)\simeq -\pi+\exp(vx)$ as $x\rightarrow -\infty$. Recall that we have defined $\varphi_v:=r_v-\pi$.
From \cite{ga} we further have $\lim_{x\rightarrow\infty}\psi(x)=0$.
Consequently, for a given $\epsilon_0>0$ there exists $x_2\in\R$ such that $2\lvert\psi(x_2)\lvert<\epsilon_0$. 
By (\ref{limit}) there thus 
exists an $v_0\in\R$ such that $\lvert \varphi_v(x_2-\log v)\lvert<\epsilon_0$ for all $v\geq v_0$.
We furthermore assume that $v_0$ is chosen such that $1+\tanh(x_2-\log v)<2\epsilon_0$ 
and $x_2-\log v_0\leq\min(x_0,x_1)$ for all $v\geq v_0$.
We chose $\epsilon_0>0$ so small that $W_-(x_2-\log v)-\tfrac{1}{2}r_v'(x_2-\log v)^2<\tfrac{1}{2}\epsilon$ for all $v\geq v_0$.
Thus we get $W_-(x_2-\log v)<\epsilon$ for $v\geq v_0$.
Since $W_{-}$ decreases strictly on the negative $x$-axis, we obtain the claim.
\end{proof}

We now show Theorem\,D, i.e., we verify that $\left(\varphi_v(x),\varphi_v'(x)\right)$ stays close to zero for bounded $x\geq d^{-}$ provided that $v$ is chosen large enough.
The proof of this result follows Lemma\,4 in \cite{BC}. 
As in \cite{BC} we introduce the distance function 
$$\rho_v:\R\rightarrow\R,\, x\mapsto\sqrt{\varphi_v(x)^2+\varphi_v'(x)^2},$$ which clearly satisfies $\rho_v>0$.

\begin{theorem}
\label{stayclose}
For any finite interval $I\subset \R$ and $\eta>0$, there exists $v_0\in\R$ such that $v\geq v_0$ implies $\rho_v(x)<\eta$ for $x\in I$.
\end{theorem}
\begin{proof}
We assume without loss of generality that $I=\lbrack x_0,x_1\rbrack$ where $x_0,x_1\in\R$ with $x_0\leq x_1$.
The ODE and $\varphi_v'(x)^2\leq \rho_v(x)^2$, $2\lvert\varphi_v(x)\varphi_v'(x)\lvert\leq \rho_v(x)^2$ imply that there exists a constant $c>0$ such that
\begin{align*}
\rho_v(x)\rho_v'(x)\leq c\rho_v(x)^2.
\end{align*}
Thus we get $\tfrac{\rho_v'(x)}{\rho_v(x)}\leq c$ and integrating this inequality from a given $T_-\leq\min(x_0,d^{-})$ to a point $x\geq T_-$ implies
\begin{align}
\label{ine}
\rho_v(x)\leq \exp(c(x-T_{-}))\rho_v(T_-).
\end{align}
Let $\epsilon>0$ be given and $x_2\in\R$ such that $1+\tanh x\leq\epsilon$ for $x\leq x_2$.
In what follows we assume that $T_-$ satisfies $T_-\leq\min(x_0,x_2,d^{-})$.
Lemma\,\ref{bigv} guarantees the existence of a velocity $v_1>0$ such that $W_{-}(T_{-})<\tfrac{1}{2}\epsilon$ for all $v\geq v_1$. 
We thus obtain
\begin{align*}
\lvert r_v'(T_{-})\lvert<\sqrt{2 \epsilon}\hspace{0.2cm}\,\mbox{and}\hspace{0.2cm}\cos^{2}(\tfrac{1}{2}r_v(T_{-}))<\tfrac{1}{\sqrt{2}}(1-\tanh T_{-})^{-\tfrac{3}{2}}{\epsilon}
\end{align*}
for all $v\geq v_1$. 
From this we get that $ r_v'(T_{-})$ becomes arbitrarily small if $\epsilon$ converges to zero. Furthermore, $ r_v(T_{-})$ becomes arbitrarily close to $-\pi$ or $\pi$.

\smallskip

Let us first assume that the latter case occurs. Hence, for any $T_+\geq\max(x_1,d^{-})$ and $\eta>0$ there exists a velocity $v_2>0$ such that $$\rho_v(T_{-})<\exp(-c(T_+-T_{-}))\eta$$ for all $v\geq v_2$. Substituting this into (\ref{ine}) yields $\rho_v(x)<\eta$ for $T_{-}\leq x\leq T_+$ and $v\geq v_0:=\max(v_1,v_2)$, whence the claim.

\smallskip

In what follows we assume that $ r_v(T_{-})$ becomes arbitrarily close to $-\pi$.
In this case we define $\hat{\varphi_v}:=r_v+\pi$ and $$\hat{\rho}_v:\R\rightarrow\R,\, x\mapsto\sqrt{\hat{\varphi}_v(x)^2+\hat{\varphi}_v'(x)^2}.$$
Similarly as above we prove that for any $x_0,x_1\in\R$ with $x_0\leq x_1$ and $\eta>0$, there exists $v_0\in\R$ such that $v\geq v_0$ implies $\hat{\rho}_v(x)<\eta$ for $x_0\leq x\leq x_1$.

\smallskip 

From the above considerations we get that for each $v\geq v_0$ we either have ${\rho}_v(x)<\eta$ or $\hat{\rho}_v(x)<\eta$ for $x_0\leq x\leq x_1$. 
Since $r_v$ depends continuously on $v\in\R$ we exactly one of the following two cases occurs
\begin{compactenum}[(i)]
\item ${\rho}_v(x)<\eta$ for all $v\geq v_0$ and for $x_0\leq x\leq x_1$; 
\item $\hat{\rho}_v(x)<\eta$ for all $v\geq v_0$ and for $x_0\leq x\leq x_1$.
\end{compactenum}
We assume that the second case occurs and choose $x_1\geq d^{+}$. 
By the proof of Theorem\,\ref{infam} there exists a velocity $v_4 \geq v_0$ such that $r_{v_4}$ is a solution of the BVP with odd nodal number.
Consequently, there has to exists a point $x_2\geq x_1\geq d^{+}$ with $r_{v_4}(x_0)=\tfrac{\pi}{2}$ and $r'(x_0)\geq 0$. Lemma\,\ref{dp} thus implies $\lim_{x\rightarrow\infty}r_{v_4}(x)=\infty$.
This contradicts the fact that $r_{v_4}$ is a solution of the BVP. Consequently, case (ii) does not occur.
\end{proof}

Note that the preceding theorem does not imply $\lim_{x\rightarrow\infty}r_v(x)=\pi$ for $v\geq v_0$!
The following corollary is a consequence of this theorem and Lemma\,\ref{limit2}.

\begin{corollary}
There exists a $v_0\in\R$ such that each solution $r_v$ of the BVP with $v_0\geq v_0$ has Brouwer degree $\pm 1$ or $\pm 3$.
\end{corollary}

\section*{Acknowledgements}
It is a pleasure to thank Wolfgang Ziller for making me aware of the Theorem of Malgrange. Furthermore, I would like to thank him for the many conservations
during the last year and for the wonderful time I had at the University of Pennsylvania.

\end{document}